\documentclass[12pt,oneside]{amsart} 
\usepackage{amssymb,amscd}
\usepackage{euscript}
\usepackage{color}

      \theoremstyle{plain}
      \newtheorem{theorem}{Theorem}[section]
      \newtheorem{lemma}[theorem]{Lemma}
      \newtheorem{corollary}[theorem]{Corollary}
      \newtheorem{proposition}[theorem]{Proposition}
      \newtheorem{remark}[theorem]{Remark}
       \newtheorem{example}[theorem]{Example}
      \newtheorem{examples}[theorem]{Examples}
      
      \newtheorem{definition}[theorem]{Definition}        
      \newtheorem{condition}[theorem]{Condition}

\numberwithin{equation}{section}

      \makeatletter
      \def\@setcopyright{}
      \def\serieslogo@{}
      \makeatother

\def\A{\EuScript{A}} 
\def\B{\EuScript{B}}
 
\def\E{\mathcal{E}}
\def\V{\mathcal{V}}

\def\M{\mathcal{M}}
\def\n{\mathcal{N}}
\def\s{\mathcal{S}} 
\def\R{\mathbb R}

\def\Z{\mathbb Z}
\def\N{\mathbb N}
\def\T{\mathbb T}

\def\Id{\text{Id}}
\def\dist{\text{dist}}
\def\va{\varphi}
\def\e{\epsilon}
\def\a{\alpha}
\def\b{\beta}

\def\QED{\hfill\hfill{\square}}

\begin{document}
 
\date{\today}
\author{Victoria Sadovskaya$^{*}$}

\address{Department of Mathematics $\&$ Statistics, 
 University of South  Alabama, Mobile, AL 36688, USA}
\email{sadovska@jaguar1.usouthal.edu}

\title [Cohomology of $GL(2,\R)$-valued cocycles]
{Cohomology of $GL(2,\R)$-valued cocycles \\
over hyperbolic systems }

\thanks{$^{*}$ Supported in part by NSF grant DMS-0901842}


\begin{abstract} 
We consider H\"older continuous $GL(2,\R)$-valued cocycles over 
a transitive Anosov diffeomorphism. We give a complete classification up to 
H\"older cohomology of cocycles with one Lyapunov exponent 
and of cocycles that preserve two transverse 
H\"older continuous sub-bundles. We prove that a measurable 
cohomology between two such cocycles is H\"older continuous.
We also show that conjugacy of periodic data for two such 
cocycles  does not always imply cohomology,
but a slightly stronger assumption does.
We describe  examples that indicate that our main results 
do not extend to general $GL(2,\R)$-valued cocycles.

\end{abstract}

\maketitle 

 
\section{Introduction}

In this paper we study cohomology of 
$GL(2,\R)$-valued cocycles over a transitive Anosov diffeomorphism 
$f$ of a compact manifold $\M$. Let  $A$ be H\"older continuous function 
from $\M$ to a  metric group  $G$.
The map $\A:\M \times \Z \to G$ defined  by 
$$
\A(x,0)=e_G,  \;\;\;\; \A(x,n) = A(f^{n-1} x)\cdots A(x), \;\text{ and }\;\;
\A(x,-n)= \A(f^{-n} x,n)^{-1} 
$$
 is called a $G$-valued 
{\em cocycle} over the $\Z$-action 
generated by $f$.  The function $A(x)=\A(x,1)$ is called the {\em generator}
of $\A$, and we will often refer to $A$ as a cocycle. 

Cocycles appear naturally in  dynamical systems,
and an important example is given by the derivative cocycle.
If the tangent bundle of $\M$ is trivial,  $T\M= \M\times \R^m$, 
then the differential $df$ can be viewed as a cocycle
$$
   \A(x,n)=df^n_x\in GL(m,\R) \quad\text{and}\quad
 A(x)=df_x.
 $$ 
More generally, one can consider the restriction of $df$
to a H\"older continuous invariant sub-bundle of $T\M$,
such as the stable or unstable sub-bundle. 
H\"older regularity of the cocycles is natural in this context, and it is also 
necessary to develop a meaningful 
theory, even in the case of $G=\R$.

\begin{definition}
Cocycles $\A$ and $\B$ are 
(measurably, continuously)  {\em cohomologous} 
if there exists a (measurable, continuous) 
function  $C:\M\to G$ such that
\begin{equation}\label{conj}
  \A(x,n)=C(f^n x) \B(x,n) C(x)^{-1} 
  \quad\text{ for all }n\in \Z \text{ and }x\in \M,
\end{equation}
equivalently, for the generators  $A$ and $B$ of $\A$ and $\B$
respectively,
$$
  A(x)=C(fx) B(x) C(x)^{-1}\quad\text{ for all }x\in \M.
$$  
\end{definition}

\noindent We refer to $C$ as a {\em conjugacy}\, between 
$\A$ and $\B$. It is also called a transfer map.
\vskip.1cm

Cocycles over hyperbolic systems and their cohomology 
have been extensively  studied starting with the seminal work 
of A. Liv\v{s}ic \cite{Liv1,Liv2},
and the research has been focused on the following questions. 

\vskip.1cm
\noindent{\bf Question 1.}
{\it Is every measurable solution $C$ of \eqref{conj} continuous?}
\vskip.1cm

Measurability should be understood with respect to a suitable 
measure, for example the measure of maximal entropy or the
invariant volume.
Further, one can ask whether a continuous solution is smooth 
if the system and the cocycles are. 

Clearly, continuous cohomology of two cocycles implies
conjugacy of their periodic data. So it is natural to ask whether 
the converse it true.

\vskip.1cm
\noindent{\bf Question 2.}
{\it Suppose that whenever $p=f^n p$,  
$\;\A (p,n)=C(p) \B (p,n) C^{-1}(p)$
for some $C(p) \in G$. 
Does it follow that $\A$ and $\B$ are continuously cohomologous?}
\vskip.1cm

Without any assumptions on continuity of $C(p)$ the answer is negative
in general. If $C(p)$ is  H\"older continuous, conjugating 
$\B$ by
$C$ reduces the question to the following.

\vskip.1cm
\noindent{\bf Question 3.}
{\it Suppose that $\A (p,n)= \B (p,n) $ whenever $f^np=p$. Does it follow that 
$\A$ and $\B$ are continuously cohomologous?}
\vskip.1cm

If $G$ is $\R$ or an abelian group, positive answers to all these 
questions where given in \cite{Liv1,Liv2}. For abelian groups, Questions
2 and 3 are equivalent, moreover, the analysis 
reduces to the case when $\B$ is the identity cocycle.
Even for non-abelian $G$, the case of $B=e_G$ has been 
studied most and by now is relatively well understood.  
In  Questions 2 and 3, the assumptions become $\A (p,n)= e_G$, and
for a Lie group $G$ these questions where answered positively by 
B. Kalinin in \cite{Ka}.  Question~1 remains open in full generality, 
but it has been answered positively under additional assumptions. 
For example, M. Pollicott and C. P. Walkden in \cite{PW} assumed 
certain pinching of the cocycle, and M. Nicol and M. Pollicott in
\cite{NP} assumed boundedness of the conjugacy.

For non-abelian $G$ the question of cohomology of two arbitrary 
cocycles is much more difficult. Positive answers to Questions 1 and 3 
were given by W. Parry \cite{P} for compact $G$ and, somewhat more 
generally, by K. Schmidt \cite{Sch} when both cocycles have ``bounded 
distortion".  The non-compact case remains largely unexplored. The only
results so far have been negative. Cocycles which are measurably but not 
continuously cohomologous were constructed in \cite{PW},
and an example of cocycles with conjugate periodic data that are not 
continuously cohomologous was given by M. Guysinsky in \cite{Guy}. 
In these examples  both cocycles can be made arbitrarily close to the 
identity, so no pinching can ensure positive results. 

\vskip.1cm

In this paper we go beyond the case of compact  groups and consider 
$G=GL(2,\R)$. We obtain positive results for two classes of cocycles.
The first one consists of cocycles which have only one Lyapunov exponent
for each ergodic $f$-invariant measure. Such cocycles can be 
identified by the periodic data: {\it for every periodic point $p=f^np$, 
the eigenvalues of the matrix $\A(p,n)$ are equal in modulus} \cite{KaS3}.
We give a complete classification of these cocycles up to 
H\"older cohomology, which shows that they can be viewed 
as either elliptic or parabolic.
At the opposite end of the spectrum is the class of cocycles which 
preserve a pair of H\"older continuous transverse sub-bundles. 
It includes uniformly hyperbolic cocycles and, more generally, 
cocycles with dominated splitting. These cocycles are H\"older
cohomologous to diagonal ones.
Using the classification we obtain positive answers to 
Questions 1 and 3 and give a complete analysis of Question 2. 
In particular, we give an example of parabolic cocycles with  
$C(p)$  uniformly bounded that are not
even measurably cohomologous.

The cocycles outside of these two classes can be viewed as non-uniformly hyperbolic. We revisit examples from \cite{Guy,PW} that give negative answers
to Questions 1 and 2 and indicate that continuous classification of such cocycles 
is unlikely. Question 3 for general $GL(2,\R)$-valued  cocycles remains open.

We would like to thank Boris Kalinin for helpful discussions.
 
 
 \section {Statement of results}

\subsection{Assumptions} \label{standing assumptions}
{\em In this paper, $\M$ is a compact connected Riemannian manifold, 
$f:\M\to\M$ is a transitive $C^2$ Anosov diffeomorphism, 
and $\E=\M\times \R^2$
is a trivial vector bundle with two-dimensional fibers.
Sub-bundles of $\E$ are understood to be  one-dimensional.
We consider orientation-preserving  H\"older continuous cocycles  
$A:\M \to GL(2,\R)$ over $f$.

Measurability is  understood with respect to a mixing 
$f$-invariant probability measure on $\M$ with   
full support and local product structure,
for example the measure of maximal entropy.
Measurable objects are assumed to be defined 
almost everywhere with respect to such a measure.
When we say that a measurable object is 
continuous, we mean that it coincides 
almost everywhere with a  continuous one.
}

\vskip.2cm

First we consider cocycles satisfying the following condition, which 
is equivalent to having only one Lyapunov exponent 
for each ergodic $f$-invariant measure
\cite{KaS3}.

\begin{condition} \label{assumptions}
For each periodic point $p=f^np$ in $\M$, 
the eigenvalues of the  matrix 
$\A (p,n)=A(f^{n-1}p) \cdots A(fp) A(p) $
are equal in modulus. 
\end{condition}

The following theorem gives a complete classification
of these cocycles up to H\"older  cohomology. 
It  shows that they can be viewed as elliptic or parabolic. 
Orientation-preserving cocycles
can have non-orientable invariant sub-bundles, 
as demonstrated  by Example  \ref{non-orientable}.
Such a sub-bundle can be made orientable by passing 
to a double cover.
For a double cover $P:\tilde \M \to \M$, 
the lift of $A$ to $\tilde \M$ defined by $\tilde A(y)=A(P(y))$.

\begin{theorem} \label{reduction}
Any cocycle $A$ satisfying Condition \ref{assumptions}
belongs to exactly one of the five types below. Cocycles of 
different types are not H\"older continuously cohomologous.

\begin{itemize}
\item[I.]  If $A$ preserves exactly one H\"older continuous
sub-bundle, which is orientable, then
$A$ is H\"older continuously cohomologous to a cocycle
$$
\hskip1cm A'(x)= k(x)\left[ \begin{array}{cc}1 & \a(x) \\  
 0 & 1 \end{array} \right], 
 \text{ where }k(x)\ne 0 \text{ and } \a \text{ is not 
 cohomologous to 0.}
$$
\vskip.1cm

\item[I$'$.] If $A$ preserves exactly one H\"older continuous
sub-bundle, which is not orientable, then there exists a cocycle
$A'$ as in I such that the lifts of $A$ and $A'$ to a double cover
are H\"older continuously cohomologous.
\vskip.1cm

\item[II.] If $A$ preserves more than one orientable H\"older 
continuous sub-bundle, then $A$ is H\"older continuously 
cohomologous to  $A'(x)= k(x)\cdot \Id$, where $k(x)\ne 0$.
\vskip.1cm

\item[II$'$.] If $A$ preserves more than one non-orientable
H\"older continuous sub-bundle, then there exists a cocycle
$A'$ as in II such that the lifts of $A$ and $A'$ to a double cover
are H\"older continuously cohomologous.
\vskip.1cm

\item[III.] If $A$ does not preserve any H\"older continuous
sub-bundles then $A$ is H\"older continuously cohomologous to
$$
  \qquad A'(x) =k(x) \left[ \begin{array}{cc} \cos \a(x) & -\sin \a(x) \\ 
        \sin \a(x)& \;\;\;\cos \a(x) \end{array} \right]
        \overset{\text{def}}{=} \,k(x) \, R(\a(x)),  \,\text{ where }k(x)>0 
 $$       
and $\a:\M\to \R/ 2\pi\Z$ is such that 
 $\a\,$mod$\, \pi$ is not cohomologous to 0 in $\R/\pi \Z$.
\end{itemize}    

\end{theorem}

We refer to cocycles $A'$ as {\em models}.
In Section \ref{sec model} we describe cohomology 
in $GL(2,\R)$ for each type of the model cocycles
giving  explicit necessary and sufficient conditions. 
In particular, we show that measurable cohomology 
between the models is H\"older. These results  together
with Theorem \ref{reduction} allow us to establish the 
following.

\begin{theorem} \label{measurable implies Holder} 
Suppose that cocycles $A$ and $B$ satisfy
 Condition \ref{assumptions}. Then
\begin{itemize}
\item[(i)] Any measurable conjugacy between  $A$ and $B$ 
is H\"older continuous;
\item[(ii)] If the diffeomorphism $f$ and the cocycles 
$A$ and $B$ are $C^k$ 
then a H\"older continuous conjugacy between 
$A$ and $B$  is $C^{r}$, where $r=k-\e$ 
for $k \in \N\setminus \{1\}$ and any $\e>0$,  and $r=k$ for $k=1,\infty, \omega$.
\end{itemize}
\end{theorem}

\begin{remark} 
Theorem  \ref{measurable implies Holder} implies that the
H\"older classification in Theorem \ref{reduction}
coincides with the measurable one.
\end{remark}

Now we consider the question whether conjugacy 
of the periodic data for two cocycles implies cohomology.

\begin{condition}  \label{nec 1}  
$A$ and $B$ have conjugate periodic data, i.e.
for every periodic point $p=f^np$ in $\M$
 there exists  $C(p)\in GL(2,\R)$ such that 
 $\A (p,n)=C(p) \B (p,n) C^{-1}(p).$
\end{condition}

\begin{proposition} \label{conf conj}
Let $A$ and $B$ be two cocycles of type II
(or III) as in 
Theorem \ref{reduction}. If $A$ and $B$ satisfy Condition 
\ref{nec 1}, then they are H\"older continuously cohomologous.

\end{proposition}

\noindent Example \ref{periodic conj}  shows
that, in general,  Condition  \ref{nec 1} does not 
imply measurable cohomology even when $C(p)$
is uniformly bounded. 

\begin{example} \label{periodic conj}
There exist cocycles 
$
A(x)=\left[ \begin{array}{cc} 1 &  \a(x) \\ 
       0& 1 \end{array} \right] \text{ and }\,
       B(x)=\left[ \begin{array}{cc} 1 &  \b(x) \\ 
 0& 1 \end{array} \right] 
$  \\
arbitrarily close to the identity that satisfy  
Condition \ref {nec 1} with $C(p)$ uniformly bounded, 
but are not measurably cohomologous. 

\end{example}

However,  continuity of the conjugacy at a single point
ensures H\"older cohomology of the cocycles.
Continuity of $C$ at $p_0$ can be replaced by a 
slightly weaker assumption that $\lim_{p\to z} C(p)$ exists
at a point $z\in \M$.

\begin{theorem} \label{periodic cont}
Let $A$ and $B$ be two cocycles satisfying 
Condition \ref{assumptions}. If $A$ and $B$ 
satisfy Condition \ref{nec 1} and $C(p)$ is continuous 
 at a point $p_0$, then $A$ and $B$ are 
H\"older continuously cohomologous.
\end{theorem}

\begin{corollary} Suppose that  cocycles $A$ and $B$  satisfying 
Condition \ref{assumptions} have the same periodic data,
i.e. $\A(p,n)=\B(p,n)$ whenever $f^np=p$. Then $A$ and $B$
are H\"older continuously  cohomologous.
\end{corollary}
\vskip.2cm

Next we consider cocycles 
that preserve two H\"older continuous transverse sub-bundles.
These include uniformly hyperbolic cocycles, and more generally 
cocycles with dominated splitting.
Such cocycles cannot be easily characterized by the periodic
data. The only positive result is due to M. Guysinsky \cite{Guy}. We recall that for a 
periodic point $p=f^n p$,  the Lyapunov exponents of 
a cocycle $A$ at $p$ are given by 
 $$
    \lambda_p=n^{-1} \ln |\lambda'_p| \quad\text{and} \quad
    \mu_p=n^{-1} \ln |\mu'_p|,
 $$
 where $\lambda'_p$ and $\mu'_p$ are the eigenvalues
 of the matrix $\A(p,n)$.

\begin{theorem}[\cite{Guy}]
Let $A:\M\to GL(2,\R)$ be a H\"older continuous cocycle over $f$.
Suppose that there exist numbers $\lambda<\mu$
and a sufficiently small $\e>0$ such that 
$$  |\lambda-\lambda_p|<\e \quad \text{and}\quad
  |\mu-\mu_p |<\e \quad \text{for every periodic point } p.
$$
Then $A$ preserves two transverse
H\"older continuous sub-bundles.
The smallness of $\e$ depends only on the map $f$, 
the numbers $\lambda$ and $\mu$, and the H\"older exponent 
of $A$.
\end{theorem}
The assumptions of the theorem are quite strong, however, it is
not sufficient just  to have $\lambda_p$ and $\mu_p$ contained in two 
disjoint closed intervals, as was demonstrated by A.~Gogolev in \cite{Gog}.

\begin{theorem} \label{2 sub-bundles}
Let $A$ and $B$ be two cocycles such that each one 
preserves  two H\"older continuous transverse sub-bundles. 
Then 
\begin{itemize}
\item[(i)] If the $A$-invariant sub-bundles are 
orientable, then $A$ is H\"older cohomologous to a diagonal cocycle.
If the sub-bundles are non-orientable, then 
there exists a diagonal cocycle $A'$ such that the lifts of 
$A$ and $A'$ to a double cover are H\"older 
cohomologous.

\item[(ii)] Any measurable conjugacy between $A$ and $B$
is H\"older continuous.

\item[(iii)] If $A$ and $B$ have conjugate periodic data,
then they are H\"older  cohomologous.

\end{itemize}

\end{theorem}

Our main results do not extend to general 
$GL(2,\R)$-valued cocycles, as demonstrated by  
the examples below, based on \cite{Guy,PW}. 
These cocycles can be viewed 
as non-uniformly hyperbolic, they have two different exponents at 
(almost) all periodic points, however the exponents can be 
arbitrarily close to each other.
The examples also indicate that  
a meaningful continuous classification of these cocycles
is unlikely due to possibility of measurable but not continuous
invariant sub-bundles.

\begin{examples} \label{general 1}
Arbitrarily close to the identity, there exist smooth cocycles 
\vskip.1cm
\noindent$
A (x)=\left[ \begin{array}{cc} \a(x) & \b \\ 0 & 1 \end{array} \right]
 \;\text{ and }\; 
B (x)=\left[ \begin{array}{cc} \a(x) & 0 \\ 0 & 1 \end{array} \right]\;
 \text{ such that} 
$
 
\begin{itemize}

\item[(i)]  $A$ and $B$ are measurably,
but not continuously cohomologous;

\item[(ii)] $A$ preserves a measurable sub-bundle 
that is not H\"older continuous;

\item[(iii)] $A$ and $B$ have conjugate periodic data,
but are not continuously cohomologous.
\end{itemize}

\end{examples}


\section{Preliminaries}

In this section we briefly introduce the main notions and results
used in this paper. 

\subsection{Anosov diffeomorphisms} 
Let $f$ be a diffeomorphism of a compact connected 
Riemannian manifold $\M$. 
It is called {\em Anosov}\, if there exist a decomposition 
of the tangent bundle $T\M$ into two invariant 
continuous sub-bundles $E^s$ and $E^u$, and constants $K>0$, 
$\kappa>0\,$ such that for all $n\in \N$, $\,{\mathbf v} \in E^s$, and 
 ${\mathbf w} \in E^u$,
$$
  \| df^n({\mathbf v}) \|  \leq K e^{-\kappa n} \| {\mathbf v} \|
     \quad\text{and}  \quad 
   \| df^{-n}({\mathbf w}) \| \leq K e^{-\kappa n} \| {\mathbf w} \|. 
$$
A diffeomorphism $f$ is called {\em transitive} if there 
exists a point in $\M$ with  dense orbit. 

The simplest  examples 
are given by {\em Anosov automorphisms}
of tori. For  a hyperbolic matrix $F$ in $SL(n,\Z)$,
the map $F:\R^n\to \R^n$ projects to an automorphism  
$f$ of the torus $\T^n=\R^n/\Z^n$,  and $f$  is clearly Anosov.
\vskip.1cm

In the rest of this section, we assume that $f$ {\it is a transitive Anosov diffeomorphism.}\, 
Abundance of periodic orbits is a key feature of 
such maps, and one of its strongest manifestations 
is  the {\it Specification Property}\, \cite[Theorem 18.3.9]{KH}:

\begin{theorem} \label{specification}
For any 
 $\e>0$ there exists a positive integer $M_\e$
such that given any collection of orbit segments
$$
  O(x_l, n_l)=\{ x_l, fx_l, \dots , f^{n_l-1}x_l\}, \quad l=1,\dots ,m,
$$
there exists a periodic point $p$ that $\e$-shadows each of the 
orbit segments with $M_\e$ iterates between consecutive ones, 
more precisely, $\,f^{n_1+\dots +n_m+mM_\e}p=p,\,$
$$
\begin{aligned}
 & \dist \,(f^i p, f^i x_1) \le \e, 
  \;\; i=0, \dots , n_1-1, \quad\text{and} \\
 & \text{for } \; l=2, \dots, m, \;\;\;
  \dist \,(f^{n_1+\dots+n_{l-1}+(l-1)M_\e+i} p, f^i x_l) \le \e, 
  \;\; i=0, \dots , n_l-1.
\end{aligned}  
  $$
\end{theorem}

We will  use  the following estimate for sums of a H\"older function
 along close orbits. This estimate is  well-known, see for example 
  \cite[Proof of Lemma 19.2.2]{KH}, and follows easily from 
  exponential closeness  of the orbit segments.

For a function  $\a:\M\to \R$, we denote
\begin{equation}\label{a(x,n)}
\begin{aligned}
   &\a^+(x,n) =\a(x)+\a(fx)+ \dots +\a(f^{n-1}x), \\
   &\a^\times(x,n) \,=\a(x)\a(fx) \cdots \a(f^{n-1}x).
\end{aligned}
\end{equation} 

\begin{lemma} \label{close orbits sum}
Let $\a:\M\to \R$ be a H\"older function with H\"older
exponent $\sigma$. Then for any  sufficiently small $\e>0$ 
there exists a constant 
$\gamma$ independent of $n$ such that  
$$
  \text{if } \; \dist\,(f^ix, f^iy)\le \e, \;\;i=0, \dots , n-1, \,\text{ then }\;\;
   \left| \a^+(x,n) -  \a^+(y,n) \right|
   \le \gamma \e^\sigma.
$$
\end{lemma}

\begin{corollary} \label{close orbits prod}
Let $\b:\M\to \R \setminus \{0\}$ be a H\"older function 
with H\"older exponent $\sigma$. Then for $x$ and $y$ 
as in Lemma \ref{close orbits sum} we have
$$
   e^{-\gamma \e^\sigma} \le  \b^\times(x,n) \cdot
   \left(  \b^\times(y,n)  \right)^{-1} \le  e^{\gamma \e^\sigma}.
$$
\end{corollary}


\subsection{Liv\v{s}ic Theorems  \cite{Liv1, Liv2}}
Let $\a:\M\to \R$ be a H\"older function. 

\begin{theorem}  \label{Liv main} 
If  $\,\a^+(p,n)=0$ whenever $f^n p=p$,
 then there exists a H\"older 
function $\va$ such that   $\a(x)=\va(fx)-\va(x)$.
Moreover,  the conclusion still holds
if  $\a^+(p,n)=0$ for every periodic point $p$ in a non-empty 
open $f$-invariant set. 
\end{theorem}
 
The stronger version was proved in \cite[Section 5]{Liv2}.
Alternatively, one can show using the Specification Property  
that the weaker assumption implies that $\a^+(p,n)=0$ for 
{\em all} periodic points.
 \vskip.1cm

 \begin{theorem} \label{Liv meas}
Let $\mu$ be an ergodic probability measure on $\M$ with   
full support and local product structure. If $\va$ is
a $\mu$-measurable function  such that   $\a(x)=\va(fx)-\va(x)$,
then $\va$ is H\"older, more precisely, $\va$ coincides 
on a set of full measure with a H\"older 
function $\tilde \va$ such that $\a(x)=\tilde \va(fx)-\tilde \va(x)$
for all $x$.
\end{theorem}

For a positive H\"older  function $\b$, Theorems \ref{Liv main} and  \ref{Liv meas} yield multiplicative counterparts:
 if $\b^\times(p,n)=1$ whenever  $f^n p=p$, then
$\b(x)=\va(fx)/\va(x)$ for a H\"older 
function $\va$; and  a measurable solution $\va$ of
the equation $\b(x)=\va(fx)/\va(x)$ is H\"older.


\subsection{Conformal structures and conformal matrices} 
\label{structure}

A conformal  structure $\s$ on $\R^2$ is a class of proportional 
inner products $\{ \langle \mathbf u, \mathbf v\rangle _\s\}$. 
It can be identified  with a real 
symmetric positive definite matrix $S$ with determinant 1 via
$$
\langle \mathbf u, \mathbf v \rangle _\s =
\langle S \mathbf u, \mathbf v\rangle ,\;\text
{ where }\langle \cdot , \cdot \rangle \text{ is the standard inner product.}
$$
The {\it standard}\,  conformal structure is given by 
$\langle \cdot , \cdot \rangle$.
The  structure $\s$ can also be viewed as a class of proportional 
ellipses $\{ E_\s\}$ given by the vectors of the same length
with respect to $\langle \mathbf u, \mathbf v\rangle _\s$.
For an invertible linear map $A:\R^2 \to \R^2$,  we denote by 
$A[\s]$ the conformal structure corresponding to the class of 
ellipses $\{ AE_\s\}$, i.e. the matrix of $A[\s]$
is $\,\det (A^*A)\cdot  (A^{-1})^*S (A^{-1})$.

Suppose that for each $x$ in $\M$ we have a conformal structure 
$\s(x)$. This defines a  conformal structure $\s$ on $\M \times \R^2$.
Let $A:\M\to GL(2,\R)$ be a cocycle. We say that $\s$
is $A$-invariant if $A(x)[\s (x)]=\s (fx)$ for all $x$.

A matrix  is called {\it conformal}\, if it preserves the
standard conformal structure, i.e it is a non-zero scalar 
multiple of an orthogonal matrix.


\section{Proof of Theorem \ref{reduction}} 

The following statement serves as a motivation
and plays an important role in our proofs.
It is an immediate  corollary of 
Propositions 2.1, 2.3, 2.6 and 2.7  in \cite{KaS3}.

\begin{proposition} \label{properties}
Suppose that a cocycle $A$ satisfies Condition
\ref{assumptions}. Then 
\begin{itemize}

\item [(i)] Any measurable $A$-invariant sub-bundle of $\E$  
 is H\"older continuous;

\item [(ii)] Any $A$-invariant measurable  conformal 
structure on $\E$ is H\"older continuous; 

\item[(iii)] The cocycle $A$ preserves either a H\"older 
continuous  sub-bundle of $\E$ or  a H\"older continuous 
conformal structure on $\E$.
\end{itemize}
\end{proposition}

The following lemma shows that  the number of 
invariant sub-bundles is an invariant of  cohomology
of the corresponding regularity. Since a continuous conjugacy
also preserves orientability of invariant sub-bundles, it follows 
that cocycles of  different types are not H\"older cohomologous.

\begin{lemma} \label{invariant}
Suppose that $A_2(x)=C(fx) A_1(x)C(x)^{-1}$ for two cocycles 
$A_1$ and $A_2$ and a measurable (H\"older) function 
$C:\M\to GL(2,\R)$.
\begin{itemize}

\item[(i)] If $A_1$ preserves a measurable (H\"older)
sub-bundle $\V_1$, then $A_2$ preserves a measurable 
(H\"older) sub-bundle $\V_2=C\V_1$.

\item[(ii)] If $A_1$ preserves a measurable (H\"older)
conformal structure $\s_1$, then $A_2$ preserves a measurable 
(H\"older) conformal structure  $\s_2=C[\s_1]$.
\end{itemize}
\end{lemma}

\noindent {\bf I-II$'$.} Suppose that  $A$ 
preserves a non-orientable  H\"older continuous sub-bundle $\V$. 

\begin{lemma} \label{double} 
There exists 
a double cover $\tilde f: \tilde \M \to \tilde \M$ of $f$ 
such that the lift $\tilde A$ of $A$ preserves an orientable 
sub-bundle $\tilde \V$ that projects to $\V$.

\end{lemma} 

\begin{proof}
We denote by $\pi_1(\M)$ the fundamental group of $\M$ and 
consider a homomorphism $\rho :\pi_1(\M) \to \Z/2\Z=\{1, -1\}$ 
defined as follows:  
$\rho(\gamma)=1$ if $\V$ 
is orientable along $\gamma$ and $-1$ otherwise.
Then $\ker \rho$ is a normal subgroup of index 2 in $\pi_1(\M)$,
and there exists a double cover $P: \tilde \M \to \M$ such that 
$P_* (\pi_1(\tilde \M))=\ker \rho$. 
The double cover has the property that the lift of a loop 
$\gamma$ in $\M$ is also a loop in $\tilde \M$ if and only if
$\V$ is orientable along $\gamma$. This property
of $\gamma$ is preserved by $f$. Indeed, the extension
$(x, v) \mapsto (fx, A(x)v)$ gives a homeomorphism 
between the restrictions of $\E$ to $\gamma$ and 
$f \circ \gamma$ which maps $\V$ to $\V$ and hence
preserves orientability. It follows that $f$ lifts to 
$\tilde f: \tilde \M \to \tilde \M$.  The  lift $\tilde \V (y) 
= \V(Py)$ of 
$\V$ to $\tilde \E =\tilde \M \times \R^2$ is orientable
and is invariant under the lift $\tilde A (y) = A(Py)$.
\end{proof}

We choose a H\"older continuous unit vector fields 
${\mathbf v}_1$ in $\tilde \V$ and $\bar {\mathbf v}_2$ 
orthogonal to ${\mathbf v}_1$.
Then for the change of basis matrix $\bar C(y)$
from the standard basis to $\{{\mathbf v}_1(y),\bar {\mathbf v}_2(y)\}$,
\begin{equation}\label{bar B}
\tilde B(y)\; \overset{\text{def}}{=}\; 
\bar C(\tilde fy)\, \tilde A(y)\, \bar C^{-1}(y)
\quad\text {is upper triangular.}
\end{equation}
Suppose that $P(y_1)=P(y_2)$ for $y_1,y_2 \in\tilde \M$.
 It follows from the construction of the 
double cover that  ${\mathbf v}_1(y_1)=-{\mathbf v}_1(y_2)$, 
hence $\bar {\mathbf v}_2(y_1)=-\bar {\mathbf v}_2(y_2)$ 
and $\bar C(y_1)=-\bar C(y_2)$. 
Thus $\tilde B(y_1)=\tilde B(y_2)$, and $\tilde B$ 
is the lift of a cocycle $ B$ on $\M$ of the  form
\begin{equation} \label{hat B}
 B(x)=k(x)\left[ \begin{array}{cc} 1 & * \\ 0 & g(x) \end{array} \right],
\quad\text{where }\;k(x)\ne 0 \;\text{ and }\; g(x)>0.
\end{equation}

The lifts $\tilde A$  of $A$ and $\tilde B$ of $ B$  are 
H\"older cohomologous via $\bar C$.  The conjugacy 
$\bar C$  does not project to $\M$ in $GL(2,\R)$, 
but does in $GL(2,\R)/\{\pm \Id\}$. It follows that $A$ and 
$B$ have the same number of invariant sub-bundles. 
Also, for any periodic point $p=f^np$ the eigenvalues 
$k^\times(p,n)$ and $k^\times(p,n)g^\times(p,n)$ of the matrix
$\B (p,n)$
have the same modulus, and hence $g^\times(p,n)=1$. 
 By Theorem \ref{Liv main}, there exists a
H\"older continuous function $\va$ such that $g(x)=\va(fx)/\va(x)$.
Rescaling the second coordinate by a factor $1/\va(x)$,
we obtain a cocycle $A'(x)$ cohomologous  to $B (x)$
 of the  form 
 \begin{equation} \label{B(x)}
  A'(x) \;=\;  k(x)\left[ \begin{array}{cc} 1 & \a(x) \\ 0 & 1 \end{array} \right],
  \quad \text{where }\; k(x)\ne 0.
 \end{equation}
The lifts $\tilde A$ of $A$ and $\tilde A'$ of $A'$ to $\tilde \M$ 
are cohomologous via a H\"older continuous function 
$\tilde C:\tilde \M \to GL(2,\R)$.
It is clear from the construction that 
\begin{equation} \label{C pm}
\text{if }\, P(y_1)=P(y_2) \,\text{ for }\,y_1\ne y_2\in \tilde \M,
\,\text{ then }\,\tilde C(y_1)=-\tilde C(y_2).
\end{equation}

\vskip.1cm

If $A$ has an orientable  invariant sub-bundle, we  obtain 
$B$ as in \eqref{hat B} without passing to a double cover. 
In this case $B$, and hence $A'$, are  H\"older 
cohomologous to $A$.

 The following lemma completes the analysis of the cases I-II$'$.

\begin{lemma} \label{number of sub-bundles}
For $A'$ as in \eqref{B(x)}
the following statements are equivalent
\begin{itemize}
\item[(i)] $A'$ preserves at least two H\"older continuous sub-bundles; 
\item[(ii)] $\a$ is  H\"older cohomologous to 0;
\item[(iii)] $A'$ is H\"older cohomologous to a scalar cocycle $k(x)\Id\,$;
\item[(iv)]  $A'$ preserves infinitely many H\"older continuous sub-bundles.
\end{itemize}
 \end{lemma}
 
 \begin{proof}
Suppose that $A'$ preserves two continuous sub-bundles.
The set  $\n \subset \M$ where the sub-bundles 
do not coincide is non-empty, 
open and $f$-invariant. Also, for every periodic point 
$p=f^np\,$ in $\n$, 
$$
  \A' (p,n)= k^\times(p,n) \left[ \begin{array}{cc} 1 & \a^+(p,n) \\
 0 & 1 \end{array} \right] \;\text{ preserves two lines,   and hence }\a^+(p,n)=0.
$$ 
 By Theorem \ref{Liv main}, $\a$ is H\"older cohomologous to 0, i.e.
 $\a(x)=s(fx)-s(x)$ for a H\"older function $s$.
If such a function $s$ exists, then

$
   A'(x)=C(fx)\cdot l(x)\,\Id \cdot C(x)^{-1}, \;
   \text{ where } C(x)= \left[ \begin{array}{cc} 1 & s(x) \\
 0 & 1 \end{array} \right],\;
$
 and (iv) follows.
 \end{proof}

We note that $A$-invariant 
sub-bundles  are either all orientable or all non-orientable.
 Indeed if a sub-bundle $\V$ is orientable, then we can 
 obtain 
 a continuous conjugacy between $A$ and $k(x)\Id$, which implies 
 that all $A$-invariant sub-bundles are orientable.

 This completes the classification 
 of cocycles that have an invariant sub-bundle.
 For future reference in the proof of the Theorem \ref{periodic cont} 
 we make the following remark.
 
 \begin{remark} \label{not sim}
 Let $A$ be a cocycle of type I$\,'$ or II$\,'$  and let $A'$ 
 be its model. Then  there is a periodic point $p=f^np$ 
 such that matrices $\A(p,n)$ and $\A' (p,n)$ are not conjugate. 
\end{remark}

\begin{proof} Let $p=f^np $ be a periodic point in $\M$ for which 
$\tilde f^n q_1 = q_2$, where $q_1$ and $q_2$ are the lifts 
of $p$. Then by \eqref{C pm} we have 
$C(\tilde f^n q_1)=\tilde C(q_2)=- \tilde C(q_1)$ and hence
$$
\begin{aligned}
\A' (p,n)= \tilde  \A' (q_1,n) =  
\tilde C(\tilde f^n q_1)\, \tilde \A (q_1,n) \, \tilde C^{-1}(q_1) =
- \tilde C(q_1)\, \A (p,n) \, \tilde C^{-1}(q_1).
\end{aligned}
$$
If $\A (p,n)$ and $\A' (p,n)$ are conjugate, so are 
$\A (p,n)$ and  $-\A (p,n)$, which is impossible.

Existence of such a point $p$ can be easily obtained.
The two lifts $\tilde f_1$ and $\tilde f_2$ of $f$ to $\tilde \M$
satisfy $\tilde f_1=  i \circ \tilde f_2$, where $i$ is the involution 
of the cover. Moreover, both lifts commute with $i$,
and hence  $\tilde f_1^n=  i^n \circ \tilde f_2^n$.
Hence for a periodic point  of an odd period $n$, 
one of the lifts has the desired property.
In fact, since both lifts have points of odd periods,
such a point $p$ exists for each lift.
 \end{proof}

\vskip.1cm


\noindent {\bf III.} Since  $A$ does not preserve any
sub-bundles,  by Proposition \ref{properties} (iii),
$A$ preserves a H\"older continuous conformal structure 
on $\E$. That is,  for every $x$  in $\M$, there is an inner product 
$\langle \cdot, \cdot \rangle _x$ such that 
$$  
\langle A(x) \mathbf u, A(x) \mathbf v \rangle _{fx} =
 k_x \langle \mathbf u, \mathbf v \rangle _x \quad\text{and}\quad
 \langle \mathbf u, \mathbf v \rangle _x =
 \langle S(x) \mathbf u,\mathbf v \rangle 
 \;\text{ for all } \mathbf u, \mathbf v \in \E_x,
 $$
where $ \langle \cdot , \cdot \rangle $ is the standard inner product and
$S(x)$ is a real symmetric positive definite matrix that depends H\"older 
continuously on $x$. For such $S(x)$ there exists a unique symmetric positive 
definite matrix $C(x)$ satisfying $S(x)=C^2(x)$, which also depends 
H\"older continuously on $x$. Then 
$\langle \mathbf u, \mathbf v \rangle _x= 
\langle S(x) \mathbf u, \mathbf v \rangle = 
\langle C(x) \mathbf u, C(x) \mathbf v \rangle$ and hence
$$
     \langle C(fx) A(x) \mathbf u, C(fx) A(x) \mathbf v \rangle =  
     \langle  A(x) \mathbf u, A(x) \mathbf v \rangle _{fx}
      = k_x \langle \mathbf u, \mathbf v \rangle _x
    =k_x  \langle C(x) \mathbf u, C(x) \mathbf v \rangle .
$$
Denoting $\,\mathbf u'=C(x)\mathbf u\,$ and $\,\mathbf v'=C(x)\mathbf v,\,$ 
we obtain
$$
 \langle C(fx) A(x) C(x)^{-1} \mathbf u', C(fx) A(x) C(x)^{-1}\mathbf v' \rangle  
   =   k_x  \langle \mathbf u', \mathbf v' \rangle
  \;\text{ for all } \mathbf u', \mathbf v' \in \E_x.
 $$
Thus $A'(x)=C(fx) A(x) C(x)^{-1}$ is H\"older continuous 
and conformal. Since $A'(x)$ is orientation-preserving,  
it is a scalar multiple of a rotation, i.e.
$A'(x)=k(x) R(\a(x))$.
Replacing $\a$ by $\a+\pi$  if necessary,
we can assume that $k$ is  positive on $\M$.

It follows from the lemma below that, since $A'$ does not 
preserve any invariant sub-bundles, the function 
$\a(x)\,$(mod $\pi$)  is not cohomologous to 0.
\vskip.1cm

\begin{lemma} \label{trivial rotation}
Let $A'(x)=k(x)\, R(\a(x))$ and $k(x)\ne 0$.
If $A'$ preserves more than one measurable conformal 
structure, then $\a\,$(mod $\pi$) is cohomologous to 0 
in $\R/\pi \Z$ and  $A'$ preserves infinitely many conformal 
structures and infinitely many sub-bundles.

\end{lemma}

\begin{proof}
 Suppose that  $A'$ preserves a
 measurable, and hence continuous, conformal structure $\s$
 different from the standard one $\s_0$. 
 The set $\n$ where $\s\ne\s_0$
is non-empty, open and invariant. At every periodic 
point $p=f^n p$ in $\n$, the matrix $\A'(p,n)=k^\times(p,n)\, R(\a^+ (p,n))$
preserves a non-circular ellipse up to scaling, and hence
$   \A'(p,n)  =  \pm k\, \Id .$
It follows that $\a^+(p,n)=0\,$(mod $\pi$) for any periodic 
point in $\n$, and  by Theorem~\ref{Liv main},
$\a\,$(mod $\pi$) is cohomologous to 0 in $\R/\pi \Z$.

Let  $\bar \a(x)= \a(x)$ (mod $\pi$) and let 
$\bar A' =k(x)\, R(\bar \a(x))$ be the projection of $A'$ to
 $GL(2,\R) / \{\pm \Id \}$. If $\bar \a$ is cohomologous to 0, i.e. 
$\bar\a(x)=\bar s(fx)-\bar s(x)$ for a  H\"older continuous 
function $\bar s:\M\to\R/\pi\Z$, 
then in $GL(2,\R) /  \{\pm \Id \}$ we have
$$
   \bar A'(x)=\bar C(fx)\cdot k(x)\,\Id \cdot \bar C(x)^{-1}, 
   \quad \text{where }
    \bar C(x) = R(\bar s(x)).
$$
Hence $\bar A'$ 
and $A'$ preserve infinitely many conformal 
structures and  sub-bundles.
\end{proof}


\section{Cohomology of the model cocycles}
\label{sec model}


First we consider cohomology of non-trivial upper triangular 
cocycles. 

\begin{proposition} \label{triangular}
 Let $\,A(x)= k(x)\left[ \begin{array}{cc}1 & \a(x) \\  
 0 & 1 \end{array} \right]$ and 
 $\;B(x)= l(x)\left[ \begin{array}{cc}1 & \b(x) \\  
 0 & 1 \end{array} \right]$,\\
where  $k(x),\,l(x)\ne 0$ for all $x$, and $\a$, $\b$
are not cohomologous to 0.  Then
\begin{itemize}
\item[(i)]   Any measurable conjugacy between $A$ and $B$
    is  H\"older  and upper triangular.
\vskip.1cm

\item[(ii)]   $A$ and $B$ are (measurably or H\"older)
cohomologous if and only if there exist H\"older functions
 $\va(x)$ and $s(x)$ and a constant  $c\ne 0$  
 such that \\
$\, k(x)/l(x)=\va(fx)/\va(x)\,$ and $\,\a(x)-c\b(x)=s(fx)-s(x)$.
\vskip.1cm

\item[(iii)] A measurable function 
$D(x)$ satisfies  $A(x)=D(fx) A(x)D(x)^{-1}$
if and only if  $\,D$ is a constant upper triangular matrix 
with equal diagonal entries.

\end{itemize} 
\end{proposition}

The last part of the proposition describes the {\it centralizer},
or the set of self-conjugacies of $A$. We discuss this set and
its connections to conjugacies in Section \ref{centralizers}.

\begin{proof} {\bf (i,$\,$ii)}
Let $C$ be a measurable 
function such that $A(x)=C(fx)B(x)C(x)^{-1}$.
We can assume that the set  where $C$ is defined is $f$-invariant. 
Clearly, $A$ preserves the sub-bundle $\E_1$ spanned by the 
first coordinate vector, and by
Lemma \ref{number of sub-bundles},
it is the only measurable invariant sub-bundle for $A$. 
Since $\E_1$ is $B$-invariant, $C(x)\E_1$ is a measurable 
$A$-invariant sub-bundle. If follows that $C(x)\E_1=\E_1$ and 
hence  the matrix $C(x)$ is upper triangular a.e.
Thus for all $x$ in an invariant set 
$X$ of full measure,
$C(x)=\va(x)${\small $\left[ \begin{array}{cc}r(x) & s(x) \\  
 0 & 1 \end{array} \right]$}.
  Then $A(x)=C(fx)B(x)C(x)^{-1}$ yields
 $$
k(x)\left[ \begin{array}{cc}1 & \a(x) \\  
    0 & 1 \end{array} \right] 
   = \frac{l(x)\va(fx)}{\va(x) }
    \left[ \begin{array}{cc} \frac{r(fx)}{r(x)} &
     -\frac{r(fx)}{r(x)}s(x)+ r(fx)\b(x)+s(fx)
     \\  0 & 1 \end{array} \right].
 $$
It follows that $k(x)/l(x)=\va(fx)/\va(x)$ on $X$. 
The functions $k$ and $l$ have constant sign
on $\M$, moreover they are of the same sign. Otherwise,
$\text{sign}\,\va(fx)=-\text{sign}\,\va(x)$ and hence for 
the sets $X_\pm=\{ x \in X : \text{sign} \, \va(x) =\pm 1\}$
we have $f(X_+)= X_-$, which contradicts mixing.
It follows from 
Theorem \ref{Liv meas} that the measurable function $\va$ 
is H\"older  and we have 
$k(x)/l(x)=\va(fx)/\va(x)$ for all $x$ in $\M$.

Since   $r(fx)/r(x)=1$,  the function $r$ is invariant, 
and hence constant a.e. 
Then
$$
 -s(x)+ c\b(x)+s(fx)= 
  \a(x), \;\text{ equivalently }\; \a(x)-c\b(x)=s(fx)-s(x) \text{ a.e.,}
$$ 
and hence by Theorem \ref{Liv meas} the measurable 
function $s(x)$ is H\"older. 

Thus, if there is a measurable conjugacy 
between $A$ and $B$ then it is of the form 
\begin{equation} \label{C triang}
        C(x)=\va(x) \left[ \begin{array}{cc} c & s(x) \\ 
        0 & 1 \end{array} \right],
 \end{equation} 
where $\va(x)$ and $s(x)$ are H\"older continuous 
functions such that  
$$ 
k(x)/l(x)=\va(fx)/\va(x) \quad \text{and}\quad \a(x)-c\b(x)= s(fx)-s(x)
\quad\text{for all }x\text{ in }\M. 
$$
Conversely, if such $c$, $\va$, and $s$ exist, then 
$C$ is a H\"older  conjugacy between 
$A$ and $B$.

\vskip.2cm

{\bf (iii)} If a measurable function $D(x)$ satisfies
$\,A(x)=D(fx) A(x)D(x)^{-1}$,
then it is of the form 
\eqref{C triang}, where 
$\,\va(fx)/\va(x)=1\,$ and  $\,(1-c)\a(x)=s(fx)-s(x).\;$
This implies that $\va$ is constant and, since $\a$ 
is not cohomologous to 0, $\,c=1$ and hence 
$s$ is constant.
Thus $D(x)=D=d${\small $\left[ \begin{array}{cc} 1 & s \\  
 0 & 1 \end{array} \right]$},
Conversely, any such matrix 
$D$ satisfies the equation.
\end{proof}


The case of scalar cocycles is simple.

\begin{proposition} \label{scalar}
Let $A(x)=k(x)\Id$ and $B(x)=l(x)\Id$, where $k(x), l(x)\ne 0$. Then
\begin{itemize}
\item[(i)] Any measurable conjugacy between $A$ and $B$ 
is of the form $\va(x)C_0$, where 
$\va(x)$ is a H\"older continuous function
such that $\va(fx)/\va(x)=k(x)/l(x)$.

\item[(ii)] A measurable function $D(x)$ satisfies  
$A(x)=D(fx) A(x)D(x)^{-1}$ if and only if  $\,D$ is
 constant.
\end{itemize} 
\end{proposition}


Now we consider non-trivial conformal cocycles.

\begin{proposition} \label{conformal}
Let  $\,A(x)=k(x) \, R(\a(x))$ and $\,B(x)=l(x) \, R(\b(x))$,\,
       where \\$k(x),\,l(x) > 0$ for all $x$, and
        $\a, \b:\M\to \R/2\pi\Z$ are such that $\a$ and $\b$ (mod $\pi$) \\
        are not cohomologous to 0 in  $\R/\pi\Z$. Then
\begin{itemize}
\item[(i)]    Any measurable conjugacy between $A$ and $B$
is H\"older  and conformal.
\vskip.05cm
\item[(ii)]   $A$ and $B$ are (measurably or H\"older)
cohomologous if and only if there exist H\"older continuous functions
 $\va:\M\to \R$ and $s:\M\to \R/2\pi\Z$ and a constant $c=\pm1$
 such that 
$\, k(x)/l(x)=\va(fx)/\va(x)\,$ and $\,\a(x)-c\b(x)=s(fx)-s(x)$.
\vskip.1cm

\item[(iii)] A measurable function $D(x)$ satisfies  
$A(x)=D(fx) A(x)D(x)^{-1}$ if and only if  $\,D(x)=D$ is a
 constant  scalar multiple of a rotation.
\end{itemize} 
\end{proposition}

It is clear from the proof that  $c=1$ and $c=-1$ in (ii) 
correspond to the conjugacy being orientation-preserving and 
orientation-reversing respectively.

\begin{proof}
{\bf (i,$\,$ii)} H\"older continuity of a measurable conjugacy can
 be obtained as a corollary 
of  the result by K. Schmidt \cite{Sch} on cocycles 
of bounded distortion. However, we will obtain 
it independently as a part of our proof.

Let $C$ be a measurable conjugacy between $A$ and $B$.
The cocycle $B$ preserves the standard 
conformal structure $\s_0$, and hence  $C[\s_0]$
is a measurable invariant conformal 
structure for $A$. By Lemma \ref{trivial rotation}, $\s_0$
is the only such conformal structure.
It follows that  $\,C[\s_0]=\s_0\,$ and hence $C$ is conformal a.e.
Since the set where $C$ is orientation-preserving
is $f$-invariant, by ergodicity, $C$ is either orientation-preserving 
a.e. or orientation-reversing a.e.
If $C$ is orientation-preserving, 
we can write  $C(x)=\va(x) R(s(x))$, 
and the equation $A(x)=C(fx)B(x)C^{-1}(x)$ yields
$$
   k(x) \,R(\a(x))  
  =(l(x) \,\va(fx)/\va(x)) \cdot R (\b(x)+s(fx)-s(x)).
$$
It follows that for almost every $x$
\begin{equation} \label{+}
    k(x)/l(x)=\va(fx)/\va(x) \quad\text{and}\quad 
    \a(x)-\b(x)= s(fx)-s(x).
\end{equation}     
By Theorem \ref{Liv meas}, the measurable functions 
 $\va$ and $s$ are H\"older continuous. 
 Conversely, if  such functions $\va$ and $s$ exist, then
 $C(x)=\va(x)R(s(x))$ is a H\"older 
 conjugacy between $A$ and $B$.
 If $C(x)$ is orientation-reversing, it is a scalar multiple 
of a reflection, 
$$
  C(x)=\va(x) \left[\begin{array}{cc}
  \cos s(x) & \;\;\; \sin s(x) \\
  \sin s(x) &  -\cos s(x) \\ \end{array}\right] 
   \overset{\text{def}}{=} \,\va(x) \, Q(s(x)).
$$
It follows that  for almost every $x$
\begin{equation}\label{-}
    k(x)/l(x)=\va(fx)/\va(x) \quad\text{and}\quad 
    \a(x)+\b(x)= s(fx)-s(x),
\end{equation}     
and hence $\va$ and $s$ are H\"older continuous.
\vskip.2cm

{\bf (iii)} 
Let $D(x)$ be a measurable function satisfying 
$A(x)=D(fx)A(x)D(x)^{-1}$. 
If $D(x)$ is orientation-preserving, then $D(x)=\va(x) R(s(x))$,
and by \eqref{+} we have $\va(fx)/\va(x)=1$ and $s(fx)-s(x)=0$.
Hence $\va(x)=d$ and $s(x)=s$ are constant, and 
$D(x)=D=d\, R(s)$. Conversely,  any such matrix 
 $D$ satisfies the equation.

If $D(x)$ is orientation-reversing,  we obtain \eqref{-} with $k=l$ and 
$\a=\b$. The latter implies that  $\a$ is cohomologous to 0, 
which contradicts the 
assumption.
\end{proof}
\vskip.2cm


\subsection
{Centralizers of cocycles and connection to conjugacies.} 
 \label{centralizers} $\;$
 
\noindent Let $A,B:\M\to G\,$ be two cocycles. 
The centralizer of $A$ is the set
 $$
    Z(A)=\{ D:\M\to G\; |\;\; A(x)=D(fx)A(x)D(x)^{-1}  \}.
 $$
 It is easy to see that 
$Z(A)$ is a group with respect to pointwise multiplication.\\
We denote by $\text{Conj}(A,B)$ the set of conjugacies 
between  $A$ and $B$, i.e.
$$
\text{Conj}\,(A,B)=
\{C\,|\;A(x)=C(fx)B(x)C^{-1}(x)\}.
$$
Both sets can be considered in any regularity.
The following properties can be verified by a direct computation.

\begin{itemize}
\item[(i)] $\text{Conj}\,(A,B)=Z(A)\cdot C,
\;\text{ where }\;C\in\text{Conj}\,(A,B).$
\item[(ii)] $Z(A)=C\cdot Z(B)\cdot C^{-1},\,$ 
 where $C\in \text{Conj}\,(A,B)$.
\end{itemize}

In  Propositions \ref{triangular},
 \ref{scalar}, and \ref{conformal} we described the centralizers
 of model upper triangular, scalar, and conformal cocycles
 respectively. In each case, the centralizer is a subgroup
  of the constant matrix functions. It follows  that a
 conjugacy between a model cocycle $A$ and a
measurable cocycle $B$ is unique up to left multiplication 
 by a constant matrix of the corresponding type. 
Property  (ii)   gives,  in particular, 
a description of the centralizer 
 of a cocycle that is cohomologous to a model one.


\section{Proofs of Theorem \ref{measurable implies Holder},
Proposition \ref{conf conj} and Theorem \ref{periodic cont}}
\label{sec same modulus}

\subsection{Proof of Theorem \ref{measurable implies Holder}} $\;$

{\bf (i)}  Lemma \ref{invariant} shows that  the number of measurable  
invariant sub-bundles is an invariant of measurable cohomology.
By Proposition \ref{properties}, measurable sub-bundles are
H\"older continuous, and it follows that cocycles with different number 
of H\"older continuous invariant sub-bundles cannot be 
measurably cohomologous. In the previous section we  
established H\"older continuity  of a measurable conjugacy for 
the three  types of model cocycles, and 
it remains to reduce the general case to the model one.

 Let $C$ be a measurable conjugacy between $A$ and $B$. 
Suppose that each cocycle  preserves exactly one 
sub-bundle, which is orientable. Then  by Theorem 
\ref{reduction}, $A$ and $B$ are H\"older 
cohomologous to model triangular cocycles 
$A'$ and $B'$. Thus we have
$$
  A'  \;\overset{C_A}{\sim}\; A \;\overset{C}{\sim} \;B
   \;\overset{C_B}{\sim}\; B'.
$$
By Proposition \ref{triangular}\,(i), the measurable conjugacy 
$C_ACC_B$ between $A'$ and $B'$ is H\"older,
 and hence so is $C$.

Suppose that $A$ and $B$ preserve unique sub-bundles, 
$\V_A$ and $\V_B$ respectively,
and at least one of the  sub-bundles is not orientable.
We pass to a double cover to make $\V_A$ orientable
and then, if necessary, we pass to a double cover 
again to make the lift of $\V_B$ orientable.
Thus we obtain lifts $\tilde A$  of $A$ and $\tilde B$ of $B$
 that are measurably conjugate via a lift $\tilde C$
 of $C$ and preserve unique sub-bundles 
 that are orientable. By the argument above,
 $\tilde C$ is H\"older continuous,
 and hence so is $C$.

The result for 
cocycles with at least two invariant sub-bundles,
and  with no invariant sub-bundles,
is obtained similarly.
\vskip.2cm

{\bf (ii)} This follows from a result by V. Ni\c{t}ic\u{a} and A. T\"or\"ok
\cite[Theorem 2.4]{NT98}. Indeed, it is easy to see that for any 
model cocycle $A$ we have
$$
   \lim_{n\to \infty} \sup_{x\in \M} \| \text{Ad}_{\A(x,n)} \|^{1/n} = 1
   \quad \text{and} \quad
   \lim_{n\to \infty} \inf_{x\in \M} \| \text{Ad}_{\A(x,n)^{-1}} \|^{-1/n} = 1,
$$
where Ad is the adjoint. It follows that the same holds for 
any cocycle  satisfying Condition \ref{assumptions}.
Hence the theorem can 
be applied with $G=GL(2,\R)$ and $\alpha_0 =0$.
$\QED$


\subsection{Proof of Proposition \ref{conf conj} } $\,$

First we consider  two model conformal cocycles $A$ and $B$ 
as in Proposition \ref{conformal} satisfying Condition \ref{nec 1}.
Since $A(x)=k(x)\, R(\a(x))$ and 
$B(x)=l(x)\, R(\b(x))$, 
$$
  \A(p,n)=k^\times(p,n)\cdot R(\a^+(p,n))
  \quad\text{and}\quad
  \B(p,n)=l^\times(p,n)\cdot R(\b^+(p,n)),
 $$ 
 which implies that $k^\times(p,n)=l^\times(p,n)$.

Suppose that  $\a^+(p,n)\ne 0\,$(mod $\pi$). Then 
$\b^+(p,n)\ne 0\,$(mod $\pi$), and both $\A(p,n)$ and
$\B(p,n)$ preserve only the standard conformal structure. 
Hence, depending on the sign of the determinant,
$C(p)$ is either a rotation or a reflection. 
In Lemma \ref{same sign} below we show that 
$\det C(p)$ has the same 
sign for all such $p$.  In the case of a rotation, 
$$
  \A(p,n)= R(s)\cdot \B(p,n)\cdot R(-s) = \B(p,n)
  \quad\text{and hence }\a^+(p,n)=\b^+(p,n).
$$
In the case of a reflection, 
$
  \A(p,n)=  Q(s)\cdot \B(p,n)\cdot Q(s) =  
   l^\times(p,n)\cdot R(-\b^+(p,n)), 
$
which implies that  $\a^+(p,n)=-\b^+(p,n)$.

 If $\a^+(p,n)=0\,$(mod $\pi$), then  $\B(p,n)=\A(p,n)$, 
 and hence $\a^+(p,n)=\b^+(p,n)=0$ or $\pi$. This implies that 
$\a^+(p,n)- \b^+(p,n)=0$ and $\a^+(p,n)+ \b^+(p,n)=0$
in $\R/2\pi\Z$.

Thus there exists a constant $c=\pm 1$ such that 
$\a^+(p,n)-c\b(p,n)=0$ whenever $f^np=p$.
Hence $k(x)/l(x)=\va(fx)/\va(x)$ and $\a(x)-c \b(x)=s(fx)-s(x)$
for H\"older  functions, and $A$ and $B$ are H\"older cohomologous.

\begin{lemma} \label{same sign}
If $C(p)$ satisfies Condition \ref{nec 1} for $A$ and $B$ 
as Proposition \ref{conformal}, then 
$\det C(p)$ has the same sign for all $p$ where 
$\a^+(p,n)\ne 0$ and $\b^+(p,n)\ne 0\,$(mod $\pi$).
\end{lemma}

\begin{proof}
Suppose that there exist two such points 
$p_1=f^{n_1}p_1$ and $p_2=f^{n_2}p_2$
with $\det C(p_1)>0$ and $\det C(p_2)<0$.
Then by the above argument, 
\begin{equation} \label{plus minus}
  \a^+(p_1,n_1)=\b^+(p_1,n_1) \quad \text{and} \quad
  \a^+(p_2,n_2)=-\b^+(p_2,n_2).
\end{equation}
We use the Specification Property, Theorem \ref{specification}.
We consider two orbit segments
\begin{equation} \label{segments}
  \{p_1, fp_1, \dots , f^{kn_1-1}p_1\} \quad\text{and}\quad 
  \{p_2, fp_2, \dots , f^{kn_2-1}p_2\}.
\end{equation}
Let $\e>0$. Then there exists $M_\e$ independent of $k$
and a periodic point $q$ such that
\begin{equation} \label{shadows}
\begin{aligned}
  &\dist \,(f^i q, f^i p_1) <\e \quad\text{for } i=0,\dots , kn_1-1, \\
 &  \dist \,(f^{kn_1+M_\e+i }q, f^i p_2) <\e 
      \quad\text{for } i=0,\dots , kn_2-1, \\
   & f^n q=q, \;\text{ where }n=kn_1+kn_2+2M_\e .  
      \end{aligned} 
\end{equation}

Let $\sigma$ be a H\"older exponent of $\a$ and $\b$,
$m_\a=\max_\M |\a(x)|$, and $m_\b=\max_\M |\b(x)|$.
Then it is easy to see using Lemma \ref{close orbits sum}
that 
$$
\begin{aligned}
 & | \a^+(q,n) - k \cdot\a^+(p_1, n_1) -k \cdot\a^+(p_2, n_2) | 
  \le  \gamma_\a \e^\sigma + 2 M_\e m_\a \quad\text{and}\\
   & | \b^+(q,n) - k \cdot\b^+(p_1, n_1) -k \cdot\b^+(p_2, n_2) | 
  \le  \gamma_\b \e^\sigma + 2 M_\e m_\b,
\end{aligned} 
$$
where constants $\gamma_\a$ and $\gamma_\b$
are independent of $k$. By \eqref{plus minus} one can choose a 
sufficiently large $k$ so that $\a^+(q,n)\ne\b^+(q,n)$ and 
$  \a^+(q,n)\ne -\b^+(q,n)$, and hence $\A(q,n)$ and $\B(q,n)$ 
are not conjugate. 
\end{proof}

This completes the proof for the case of two model 
conformal cocycles.
 Suppose that $A$ and $B$ are two cocycles of type III
 satisfying Condition \ref{nec 1}. Then by  
 Theorem~\ref{reduction}, $A$ and $B$ are H\"older cohomologous 
 to model conformal cocycles $A'$ and $B'$, 
and it is easy to see that $A'$ and $B'$
also satisfy Condition \ref{nec 1}.  Hence $A'$ and $B'$ 
are H\"older 
cohomologous, and so are $A$ and $B$.
The result for cocycles of type II follows  similarly from 
Theorem  \ref{reduction} and Proposition \ref{scalar}.
$\QED$


\subsection{Proof of Theorem \ref{periodic cont}}

In the rest of this section, we consider

\begin{condition}  \label{nec 2}  
For every periodic point $p=f^np$
 there exists  $C(p)\in GL(2,\R)$ such that 
 $\A (p,n)=C(p) \B (p,n) C^{-1}(p)$,
 and $C(p)$ is continuous at a point $p_0$.
\end{condition}

\begin{proposition} \label{periodic triang}
Let $A$ and $B$ be triangular cocycles as in 
Proposition \ref{triangular}. \\ If $A$ and $B$  satisfy  
Condition \ref {nec 2}, 
then $A$ and $B$ are H\"older cohomologous.
\end{proposition}

\begin{proof}
For every periodic point $p=f^np$ we have
 $$
    \A (p,n)= k^\times(p,n)\left[ \begin{array}{cc} 1 & \a^+(p,n) \\  
 0 & 1 \end{array} \right] \;\text{ and}\quad
   \B (p,n)= l^\times(p,n)\left[ \begin{array}{cc} 1 & \b^+(p,n) \\  
 0 & 1 \end{array} \right].
$$
Since the  matrices are conjugate,  $k^\times (p,n)=l^\times (p,n)$.
Below we show that there exists a constant $c$
such that $\a^+(p,n)=c\b^+(p,n)$ whenever $f^np=p$.
By Theorem~\ref{Liv main}, $k(x)/l(x)=\va(fx)/\va(x)$
and $\a(x)-c\b(x)=s(fx)-s(x)$ for H\"older
functions $\va$ and $s$, and hence $A$ and $B$ are 
H\"older cohomologous by Proposition \ref{triangular}.

\vskip.1cm

Suppose that there exist points $p_1=f^{n_1}p_1$
and $p_2=f^{n_2}p_2\,$ such that 
\begin{equation} \label{c1 and c2}
\a^+(p_{1,2},n_{1,2})\ne 0,\;\; \b^+(p_{1,2},n_{1,2})\ne 0, 
  \quad \frac{\a^+(p_1,n_1)}{\b^+(p_1,n_1)}=c_1 \ne c_2=
   \frac{\a^+(p_2,n_2)}{\b^+(p_2,n_2)}.
\end{equation}
Let $z\in \M$ and  $\e>0$.
We consider two orbit segments: $\{z\}$ and 
$\{p_1, fp_1, ...  f^{kn_1-1}p_1\}$. By the Specification Property,
 there exists  a number $M_\e$ independent of $k$
and  a point $q_1= f^{t_1}q_1$, where $t_1=kn_1+2M_\e+1$, 
such that
$$
    \dist (q_1,z)<\e \quad\text{and}\quad
    \dist (f^{M_\e+1+i}q_1,\, f^i p_1) <\e
    \;\text{ for }i =0, \dots, kn_1-1.
$$
\vskip.1cm

Let $\sigma$ be a H\"older exponent of  $\a$ and $\b$, 
$m_\a =\max_\M |\a(x)|$ and $m_\b =\max_\M |\b(x)|$.
It follows easily from Lemma \ref{close orbits sum} that 
there exists constants
$\gamma_\a$ and $\gamma_\b$ independent of $k$ such that
$$
\begin{aligned}
&| \, \a^+(q_1, t_1) -  \,k \cdot
  \a^+(p_1,n_1) \, |
 \le \gamma_\a \e^\sigma + (2M_\e+1)m_\a 
 \;\;\text{ and}\\
&| \,  \b^+(q_1, t_1)  -  \,k \cdot
  \b^+(p_1,n_1) \, |
 \le \gamma_\b \e^\sigma + (2M_\e+1)m_\b.
\end{aligned}
$$
Let $\delta=\frac13 |c_1-c_2|$. 
Since $\a^+(p_1,n_1)\ne 0$ and $\b^+(p_1,n_1)\ne 0$, 
by choosing a sufficiently large $k$,
we can ensure that $\a^+(q_1, t_1)\ne 0$,
$\,\b^+(q_1, t_1)\ne 0$, and
$$
\left| \, \frac{\a^+(q_1, t_1)}{\b^+(q_1, t_1)}-
c_1\,\right|  = \left| \,
\frac{\a^+(q_1,t_1)}{\b^+(q_1, t_1)}-
\frac {k\cdot \a^+(p_1,n_1)}{k \cdot  \b^+(p_1,n_1)} \,\right| 
<\delta.
$$

Similarly, using the orbit of $p_2$,
in an $\e$-neighborhood of $z$ we can find a periodic 
point $q_2$ of a period $t_2$ such that 
$\a^+(q_2, t_2)\ne 0$, 
$\b^+(q_2, t_2)\ne 0$, and the ratio
$\a^+(q_2, t_2)/\b^+(q_2, t_2)$ is $\delta$-close to $c_2$.
It follows that at every point $z\in \M$, the ratio 
$\a^+(p,n)/\b^+(p,n)$ has no limit, 
and in particular the ratio is discontinuous
at every periodic point.
\vskip.1cm

Suppose that $\a^+(p,n)\ne 0$ and $\b^+(p,n)\ne 0$ 
and $\A(p,n)=C(p)\B(p,n)C(p)^{-1}.$   
Then $C(p)$ is upper triangular, and a direct calculation 
shows that it is of the form
  $$
    C(p)= \va(p) \left[ \begin{array}{cc} 
    \a^+(p,n)/\b^+(p,n) & d(p) \\ 0 & 1
   \end{array} \right]. 
$$
Therefore, discontinuity of the ratio $\a^+(p,n)/\b^+(p,n)$
implies discontinuity of $C$.

Thus no two periodic points satisfy \eqref{c1 and c2}.
It follows  that there exists a constant $c$ 
such that $\a^+(p,n)=c\b^+(p,n)$ at every periodic point,
and hence $A$ and $B$ are H\"older cohomologous.
\end{proof}
\vskip.2cm


Next we show that cocycles $A$ and $B$ of different types, 
as in Theorem \ref{reduction}, cannot satisfy Condition \ref{nec 2}. 
Clearly, this is the case for cocycles with different number
of invariant sub-bundles. The following lemma establishes
this for cocycles with different orientability types of
invariant sub-bundles.

\begin{lemma} \label{not conjugate}
Let $A$ be a cocycle of type I (II) and $B$ be a cocycle
of type I$\,'$ (II$\,'$). Then $A$ and $B$ do not satisfy 
Condition \ref{nec 2}.
\end{lemma}

\begin{proof} 
Suppose that cocycles $A$ of type I and $B$  of type I$'$ 
satisfy Condition \ref{nec 2}. By Theorem \ref{reduction},
there exist  model triangular cocycles $A'$ and $B'$ such that
$A'$ is H\"older cohomologous to $A$, 
and the lifts $\tilde B$ of $B$ and $\tilde B'$ of $B'$
to a double cover 
$\tilde \M$ are H\"older cohomologous. 
Clearly, the lifts $\tilde A$ of $A$ and $\tilde A'$ of $A'$ to 
$\tilde \M$ are also H\"older  cohomologous.
Since $A$ and $B$ satisfy Condition \ref{nec 2}, 
so do $\tilde A$ and $\tilde B$ and hence  
the model cocycles $\tilde A'$ and $\tilde B'$.
By Proposition \ref{periodic triang}, the cocycles
$\tilde A'$ and $\tilde B'$ are H\"older cohomologous, 
and it follows from Lemma \ref{model conj} below that
so are $A'$ and $B'$.
Thus the cocycle $B$ of type I$'$ and its 
model $B'$ have conjugate periodic data,
which contradicts Remark \ref{not sim}. 
A similar argument yields the result for cocycles of types 
II and II$'$.
\end{proof}

\begin{lemma} \label{model conj} 
Let $A$ and $B$ be model triangular cocycles as 
in Proposition \ref{triangular} and let $\tilde A$ and 
$\tilde B$ be their lifts to the same double cover. 
Then a H\"older conjugacy between $\tilde A$ and 
$\tilde B$ projects to a H\"older conjugacy
between $A$ and $B$.

\end{lemma}

\begin{proof} We denote the lifts of $\a$ and  $\b$ 
by $\tilde\a$ and  $\tilde \b$. 
Since $\tilde A$ and $\tilde B$ are H\"older cohomologous,  
$\tilde \a^+(q,m)=c\tilde \b^+(q,m)$
whenever $\tilde f^m q=q \in \tilde \M.$
Let $p=f^{n}p\in \M$ and  let  $q\in \tilde \M$ be such that $p=P(q)$.   
Then $q=\tilde f ^{m}(q)$ where $m$ is either $n$ or $2n$, 
and  it follows that  $\a^+(p,n)=c\b^+(p,n)$.
Similarly,  $k^\times(p,n)=l^\times(p,n)$
whenever $f^np=p$. Thus $A$ and $B$ are H\"older 
cohomologous.
The discussion in Section \ref{centralizers} implies
that a conjugacy between $A$ and $B$, as well as  between 
$\tilde A$ and $\tilde B$, is unique up to 
multiplication by a constant  matrix.
Hence a conjugacy between $A$ and $B$ is the projection
of a conjugacy between $\tilde A$ and $\tilde B$.
\end{proof}

We conclude that if $A$ and $B$ satisfy Condition \ref{nec 2},
then they are of the same type. By Proposition \ref{conf conj}, 
it remains to consider cocycles of types I, I$'$, and II$'$.
If $A$ and $B$ are of type I, they are H\"older cohomologous to 
model triangular cocycles $A'$ and $B'$, respectively. It follows 
that  $A'$ and $B'$ also satisfy Condition \ref{nec 2}. By  
Proposition~\ref{periodic triang}, $A'$ and $B'$ 
are H\"older cohomologous, and hence so are $A$ and $B$.

Let $A$ and $B$ be cocycles of type I$'$. It follows easily 
from Lemma \ref{not conjugate} that their invariant sub-bundles
can be made orientable by passing to the same double cover $\tilde \M$.
The lifts  $\tilde A$ and $\tilde B$  to  $\tilde \M$
are of type I and   satisfy Condition \ref{nec 2}, 
 and hence  are 
H\"older cohomologous.
 By Theorem \ref{reduction},
there exist model triangular cocycles $A'$ and $B'$
whose lifts to $\tilde \M$ are H\"older cohomologous 
to  $\tilde A$ and $\tilde B$  respectively. Thus we have
$$
\begin{array}{ccccccc}
\tilde A' &  \overset{\tilde C_A}{\sim} & \tilde A & \overset{\tilde C}{\sim} 
               &  \tilde B &  \overset{\tilde C_B}{\sim}   & \tilde B' \\
 \downarrow  &    &  \downarrow  &    
           &  \downarrow  &  & \downarrow  \\
 A' &  &  A &  &  B &  &  B'  
\end{array}
$$
Let  $y_1,y_2\in \tilde \M$ be such that $P(y_1)=P(y_2)$. By \eqref{C pm}, 
$\;\tilde C_A(y_1)=-\tilde C_A(y_2)$ and $\tilde C_B(y_1)=-\tilde C_B(y_2)$.
By Lemma \ref{model conj},
the conjugacy $C'=\tilde C_A\tilde C\tilde C_B$ between $\tilde A'$ and 
$\tilde B'$ projects to $\M$, which means that 
$C'(y_1)=C'(y_2)$. Thus, $\tilde C(y_1)=\tilde C(y_2)$ and hence 
$\tilde C$ projects to a conjugacy between $A$ and $B$.

A similar argument yields the result for cocycles of type II$'$.
$\QED$


\section{Proof of Theorem \ref{2 sub-bundles}}
\label{sec diff moduli}

\vskip.2cm

First we discuss cohomology of diagonal cocycles
that serve as models for cocycles with two transverse 
invariant sub-bundles. We denote the coordinate sub-bundles
by $\E_1$ and $\E_2$, and we denote a diagonal matrix with 
entries $\a_1, \a_2$ by $\text{diag}\,(\a_1, \a_2)$.

\begin{lemma} \label{E1,E2}
Let $A(x)=\text{diag}\,(\a_1(x), \a_2(x))$, where 
 $\a_{1,2}(x)\ne 0$. 
 $A$ preserves a measurable sub-bundle 
other than $\E_1$ and $\E_2$   if and only if 
$\a_1(x)/\a_2(x)=s(fx)/s(x)$ for a H\"older function $s(x)$,  equivalently,
$A(x)$ is H\"older cohomologous to  $\a_1(x)\Id$.
\end{lemma}

\begin{proof}

Let $\V$ be the measurable invariant sub-bundle.
Since the set where $\V$ differs from $\E_1$ and $\E_2$ is invariant,
it is of full measure by ergodicity. Therefore we can write 
$\V(x)$ as the span of 
${\mathbf v}(x)=${\tiny $\left[ \begin{array}{c} s(x) \\ 1 \end{array} \right]$},
where $s$ is a non-zero measurable function.
Then  ${\mathbf v}(fx)=c(x)\cdot A(x) {\mathbf v}(x)$ for a scalar function $c$,
which implies that $\a_1(x)/\a_2(x)=s(fx)/s(x)$ and  
hence $s$ is H\"older. It follows that 
$$
  \a_1(x)\,\Id =C_A(fx) \cdot A(x) \cdot C_A(x)^{-1}, 
  \quad\text{where}\quad C_A(x)= \text{diag}\,(1, s(x)).
$$
Clearly, if $A$ is H\"older cohomologous to a scalar cocycle,
then $A$ preserves infinitely many H\"older continuous sub-bundles.
\end{proof}

Cocycles cohomologous to scalar ones were discussed 
 in the previous sections.

\begin{proposition} \label{diagonal}
Suppose $A(x)=\text{diag}\, (\a_1(x), \a_2(x))$ and 
$\,B(x)=\text{diag}\, (\b_1(x), \b_2(x))$, 
where $\a_{1,2}(x)\ne 0$ and $ \b_{1,2}(x)\ne 0$,
are not cohomologous to scalar cocycles. Then
\begin{itemize}

\item[(i)] $A$ and $B$ are H\"older  cohomologous
if and only if there exist measurable, \\ equivalently H\"older, 
functions $s_1$ and $s_2$ such that either \\
$\a_1(x)/\b_1(x)=s_1(fx)/s_1(x)$ and $\a_2(x)/\b_2(x)=s_2(fx)/s_2(x)$
for all $x$, or\\
$\a_1(x)/\b_2(x)=s_1(fx)/s_1(x)$ and $\a_2(x)/\b_1(x)=s_2(fx)/s_2(x)$
for all $x$.
\vskip.1cm

\item[(ii)] Any measurable conjugacy  between $A$ and
$B$ is H\"older  and either diagonal or anti-diagonal.

\item[(iii)] The centralizer of $A$ consists of all constant diagonal matrices.

\item[(iv)] If $A$ and $B$ have conjugate periodic data,
then they are H\"older cohomologous.

\end{itemize}

\end{proposition}

\begin{proof}

{\bf (i,$\,$ii,$\,$iii)}  If $A(x)=C(fx)B(x)C(x)^{-1}$
for a measurable  function $C$, then  measurable 
sub-bundles $C(\E_1)$ and $C(\E_2)$ are $A$-invariant.
It follows from Lemma \ref{E1,E2} that 
 either $C(\E_1)=\E_1$ and $C(\E_2)=\E_2$, or
$C(\E_1)=\E_2$ and $C(\E_2)=\E_1$. Therefore, 
$C(x)$ is either a diagonal or an
anti-diagonal matrix. This reduces the questions of cohomology 
of $A$ and $B$ 
to that of cohomology of the scalar functions 
$\a_i$ and $\b_i$, and (i), (ii), and (iii) follow easily.

\vskip.2cm

{\bf (iv)}
By (ii) it suffices to show that either 
\begin{equation} \label{a,b}
\begin{aligned} 
  & \a_1^\times (p,n)=\b_1^\times (p,n) \;\text{ and }\;
  \a_2^\times (p,n)=\b_2^\times (p,n) \;\text{ whenever }
  f^np=p, \;\text{ or} \\
  & \a_1^\times (p,n)=\b_2^\times (p,n) \;\text{ and }\;
  \a_2^\times (p,n)=\b_1^\times (p,n) \;\text{ whenever }\;
  f^np=p.
  \end{aligned}
\end{equation}
At every point $p=f^n p\,$ the  eigenvalues  of 
  $\A(p,n)$ and  $\B(p,n)$  are equal, i.e. 
\begin{equation} \label{a_1,a_2}
      \{\,\a_1^\times(p,n), \; \a_2^\times(p,n) \,\}=
       \{\,\b_1^\times(p,n), \; \b_2^\times(p,n) \,\}.
\end{equation}
Suppose that for points $p_1=f^{n_1}p_1$ and 
$p_2=f^{n_2}p_2$ we have 
$$
   \a_1^\times(p_1,n_1)=\b_1^\times(p_1,n_1)\ne \b_2^\times(p_1,n_1)
   \quad \text{and}\quad 
   \a_1^\times(p_2,n_2)= \b_2^\times(p_2,n_2)\ne \b_1^\times(p_2,n_2).
$$
We proceed 
as in the proof of Lemma~\ref{same sign}.
We consider orbit segments as in \eqref{segments}
and a periodic point $q=f^nq$ satisfying
\eqref{shadows}. By Corollary~\ref{close orbits prod}
there exist constants $\gamma$ and $\sigma$ 
independent of $k$ such that 
$$
 e^{-\gamma\e^\sigma} \le 
 \frac{\a_1^\times(q, kn_1)}{\a_1^\times(p_1, kn_1)} \le  
 e^{\gamma\e^\sigma}, \quad\;\;
  e^{-\gamma\e^\sigma} \le 
 \frac{\a_1^\times(f^{kn_1+M_\e}q, kn_2)}{\a_1^\times(p_2, kn_2)} \le  
 e^{\gamma\e^\sigma},
$$
and the same estimates hold with $\b_1$ in place of $\a_1$.
Taking a sufficiently large $k$ ensures  that 
$\a_1^\times(q,m)\ne\b_1^\times(q,m)$. 
A similar argument shows that $\a_1^\times(q,m)\ne \b_2^\times(q,m).$
This  contradicts \eqref{a_1,a_2} and hence
\eqref{a,b} is satisfied.
\end{proof}

Now we complete the proof of the theorem. 
We consider a cocycle $A$ with two H\"older continuous 
transverse sub-bundles $\E^1_A$ and $\E^2_A$.
Example (iii) in Section \ref{non-orientable} shows that the  
sub-bundles are not necessarily orientable.
Clearly, if  one of the two invariant sub-bundles 
is orientable, then so is the other one.

(i)
If the $A$-invariant sub-bundles are orientable, then there exist 
continuous unit vector fields ${\mathbf v}^1$ and ${\mathbf v}^2$ 
spanning $\E^1_A$ and $\E^2_A$ respectively. 
Let $C_A(x)$ be the change of basis matrix from 
$\{ {\mathbf v}^1(x),{\mathbf v}^2(x)\}$ to the standard basis. 
Then $C_A$ is H\"older continuous and 
the cocycle $A'(x)=C_A(fx)A(x)C_A(x)^{-1}$ is diagonal.

If $\E^1_A$ is not orientable, using a double cover 
as in Lemma \ref{double} we obtain as in the proof
of Theorem \ref{reduction} a cocycle $A''$ with two transverse 
orientable invariant sub-bundles
such that the lifts of $A$ and $A''$ are H\"older  cohomologous.
Then $A''$ is cohomologous to a diagonal cocycle as before.

Now (ii) and (iii) follow from (i) and Proposition \ref{diagonal}
as in the proofs of Theorems \ref{measurable implies Holder} and 
\ref{periodic cont} respectively. We note that conjugacy 
of the periodic data implies conjugacy of the model diagonal
cocycles and precludes having invariant sub-bundles of
different types as in Lemma \ref{not conjugate}.
$\QED$

\vskip.3cm

\section{Examples}


\subsection{Orientation-preserving cocycles with 
non-orientable invariant sub-bundles} \label{non-orientable}

{\it There exists a smooth orientation-preserving 
cocycle $A$ such that
\begin{itemize}
\item[(i)] $A$ has a unique invariant  sub-bundle that is not orientable;
\item[(ii)] $A$ preserves infinitely many non-orientable sub-bundles;
\item[(iii)] $A$ preserves exactly two transverse sub-bundles, which are non-orientable.
\end{itemize}
}
\vskip.2cm
\noindent 
Let $\T^2=\R^2 / \Z^2$ be the standard torus and let 
$\tilde \T^2=\R^2 / (2\Z \times \Z)$ be its double cover.
We consider 
$\,F= ${\small $\left[ \begin{array}{cc}  5 & 2 \\  2 & 1 \end{array} \right]$},\,
or any hyperbolic  matrix $[F_{ij}]\,$ in $SL(2,\Z)\,$
such that $F_{11}$ is odd and $F_{12}$ is even.
The map $F:\R^2 \to \R^2$ projects to Anosov automorphisms
$f:\T^2 \to \T^2$ and $\;\tilde f: \tilde \T^2 \to \tilde \T^2$.
Let $\tilde C: \tilde \T^2\to GL(2,\R)$ be the function given by 
$\tilde C(x)=R(\pi x_1)$, the rotation by the angle $\pi x_1$. 
This function is  not 1-periodic in $x_1$, and hence
it does not project to $\T^2$. 

\vskip.2cm

{\bf (i)} We define a cocycle $\tilde A:\tilde \T^2 \to GL(2,\R)$ 
over $\tilde f$ as 
\begin{equation}\label{tilde A}
   \tilde A(x)=\tilde C(\tilde fx)\, B\, \tilde C(x)^{-1}, \quad \text{where }\;
   B=  \left[ \begin{array}{cc}  1 & 1 \\  0 & 1\end{array} \right].
\end{equation}
The calculation below shows that  $\tilde A$ is 1-periodic in both 
$x_1$ and $x_2$ and thus it projects to a continuous and, in fact, 
analytic cocycle $A$ over $f$ on $\T^2$.
$$ 
\begin{aligned}
    & \tilde A(x)= R(\pi(5x_1+2x_2)) 
    \left(\,\Id +\left[ \begin{array}{cc}  0 & 1 \\  0 & 0 \end{array} \right]
    \right)
     R(-\pi x_1) = \\
     &=R(2\pi(2x_1+x_2))+
     \left[ \begin{array}{cc} 
     \frac{\sin (2\pi(2x_1+x_2))+\sin(2\pi(3x_1+x_2))}{2} & 
     \frac{\cos (2\pi(2x_1+x_2))+\cos(2\pi(3x_1+x_2))}{2}  \\ 
      \frac{\cos (2\pi(3x_1+x_2))-\cos(2\pi(2x_1+x_2))}{2}  & 
       \frac{\sin (2\pi(2x_1+x_2))+\sin(2\pi(3x_1+x_2))}{2} \end{array} \right].  
    \end{aligned}
$$

The constant cocycle $B:\tilde \T^2\to GL(2,\R)$ preserves
only the  sub-bundle $\tilde \E_1$ spanned by the the first coordinate vector. 
Hence $\tilde \V(x)=\tilde C(x)\,\tilde \E_1$ is the unique invariant 
 sub-bundle for $\tilde A$.
As $\tilde C(x)=R(\pi x_1)$, it is easy to see that
$\tilde \V$ projects to a continuous $A$-invariant sub-bundle $\V$
on $\T^2$, which is not orientable as its orientation  
is reversed along the first coordinate loop.

Clearly, the cocycles $A$ and $B$ on $\T^2$
are not continuously cohomologous as a continuous
conjugacy preserves orientablility of invariant
sub-bundles. In fact, they are not even measurably
cohomologous, as follows from Theorem 
\ref{measurable implies Holder}. 
However, their lifts  are smoothly 
cohomologous via $\tilde C$ on $\tilde \T^2$.

\vskip.2cm

{\bf (ii)} Considering $B=\Id\,$ in \eqref{tilde A} yields an example.
Since any constant sub-bundle $\tilde \V_{\text{const}}$
is preserved by $B$, 
$\tilde A$ preserves the sub-bundles 
$\tilde \V=\tilde C \tilde \V_{\text{const}}$.
These sub-bundles project to $A$-invariant non-orientable
sub-bundles on $\T^2$.

\vskip.2cm

{\bf (iii)} We consider
$B= \text{diag}\,(2,1)$ 
in \eqref{tilde A}.  It is easy to see
that then the cocycle $A$ has exactly two transverse invariant sub-bundles 
that  are non-orientable.
$\QED$

\vskip.3cm
\subsection{Construction of Example  \ref{periodic conj}}$\;$

\noindent Let  $\a$ and $\b$ be H\"older functions
such that  $\a(x)>0$ and  $\b(x)>0$ for all $x$ in $\M$;\\
for two periodic points $p_1$ and 
$p_2$ of periods $n_1$ and $n_2$ respectively,
$$
\a(f^ip_1)=\b(f^i p_1), \;\; 0\le i \le n_1-1,\quad\text{and}\quad
\a(f^ip_1)=2\b(f^i p_2), \;\; 0\le i \le n_2-1;
$$ 
and
$\b(x)\le \a(x) \le 2\b(x)<\e$ for all $x$.
The function $\b$ can be chosen constant.
\vskip.1cm

 Since $\a^+(p,n)> 0$ and 
$\b^+(p,n)> 0$ at every periodic point $p$, 
 the functions $\a$ and $\b$ are not cohomologous to 0,
 and the matrices $\A(p,n)$ and $\B(p,n)$ are conjugate by 
$$
   C(p)=\left[ \begin{array}{cc}  \a^+(p,n)/\b^+(p,n) &  0 \\ 
       0& 1 \end{array} \right].
$$
Since $1\le  \a^+(p,n)/\b^+(p,n) \le 2$
for every $p$, $\,C(p)$ is uniformly bounded.

As $\a^+(p_1,n_1)=\b^+(p_1,n_1)$ and
 $ \a^+(p_2,n_2)=2\b^+(p_2,n_2),$
there is no constant $c$ such that
$\a^+(p,n)-c\b^+(p,n)=0$ for every periodic $p$.
Thus by Proposition \ref{triangular}\,(ii) the cocycles $A$ 
and $B$ are not measurably cohomologous.
$\QED$


\subsection{Construction of Examples \ref{general 1}} $\;$
\vskip.2cm

{\bf (i)} We describe a simplified version of the example in 
Section 9 of \cite{PW}. \\
Let $f:\M\to \M$ be a $C^2$ Anosov diffeomorphism that 
fixes a point $x_0$,  and let $\a(x)$ be a smooth function such 
that $\a(x_0)=1$ and  $0<\a(x)<1$ for all $x\ne x_0$.  
The cocycles can be made arbitrarily close to the identity
by choosing $\b$ close to 0 and $\a(x)$ close to 1.
Since the matrices $A(x_0)$ and $B(x_0)$
are not conjugate,  the cocycles $A$ and $B$
are not continuously cohomologous.

A measurable conjugacy is constructed in the form 
$C(x)=${\small $\left[ \begin{array}{cc} 1 & c(x) \\ 0 & 1 \end{array} \right]$}.
Then $A(x)=C(fx)B(x)C(x)^{-1}$ is equivalent to 
$$
\left[ \begin{array}{cc} \a(x) & \b \\ 0 & 1 \end{array} \right] =
\left[ \begin{array}{cc} \a(x) & c(fx) - \a(x)c(x) \\ 0 & 1 \end{array} \right].
$$
A function $c\,$ such that $\,c(fx) =\b+ \a(x)c(x)$ is obtained as 
a series.  Let
$$
  c_m(x)=\b\cdot \left( 1+\a(f^{-1}x)+\a(f^{-1}x)\a(f^{-2}x)+ \cdots +
  \a(f^{-1}x)\cdots \a(f^{-m}x) \right).
$$  
By  the Birkhoff Ergodic Theorem,
$( \a(f^{-1}x)\cdots \a(f^{-m}(x) ) ^{1/m} \to \a<1$ a.e.
It follows that the sequence $\{c_m(x)\}$ converges to a limit 
$c(x)$ a.e., and the function $c(x)$ is measurable 
as a limit of continuous functions. The functions $c_m$
satisfy the equation $c_m(fx)=\b+\a(x)c_{m-1}(x)$, and passing 
to the limit we see that $c(fx)=\b+\a(x)c(x)$.
   
\vskip.2cm
{\bf (ii)} Let $f$ and $\a$ be as in  (i). Clearly, $B$ preserves 
the coordinate sub-bundles $\E_1$ and $\E_2$.  Hence $A$  preserves 
$\E_1$ and a measurable sub-bundle $ \V =C \E_2\ne \E_1$, which is not continuous. Indeed, as we show below $\E_1$ is the only continuous 
$A$-invariant sub-bundle.
A direct calculation shows that  for $p=f^np$,

$$
\begin{aligned}
\A(p,n)& =\left[ \begin{array}{cc}  \a^\times(p,n) & \b\cdot \a^*(p,n) 
\\ 0& 1  \end{array} \right], \; \text{ where }\\
    \a^*(p,n) & =1+\a(f^{n-1}p)+\a(f^{n-1}p)\a(f^{n-2}p)+\cdots +  
    \a(f^{n-1}p)\cdots \a(fp) =\quad \\ 
    & =1+\a(f^{-1}p)+\a(f^{-1}p)\a(f^{-2}p)+\cdots +  
    \a(f^{-1}p)\cdots \a(f^{-n+1}p).
\end{aligned}
$$
Hence the eigenvectors of the matrix $\A(p,n)$ are ${\mathbf e}_1$ and 
\begin{equation} \label{v(p)}
 {\mathbf v}(p)=\left[ \begin{array}{c}  
c(p) \\ 1 \end{array} \right], \quad\text{where}\quad 
c(p)= \frac{\b\cdot \a^*(p,n)}{1- \a^\times(p,n)}\,.
\end{equation}
We note that $c(p)=\lim_{m\to\infty} c_m(p)$, where $c_m(p)$
are as in (i).

\begin{lemma} \label{c(p)}
Let $\,x\ne x_0$,  $\,\e>0$, and $\,N>0$.
Then there exists a periodic point $q\ne x_0$
such that $\,\dist\,(q,x)< \e\,$ and $\,c(q)>N$.

\end{lemma}

\begin{proof}
We assume that $0<\e<\frac12 \dist(x,x_0)$
and apply  the Specification Property 
to the  orbit segments $\{x\}$ and 
$\{x_0, fx_0, \dots, f^{k-1} x_0 \}=\{x_0, \dots, x_0\}$. 
Then there exists  a number $M_\e$ independent of $k$ and 
a periodic point $q$ such that 
$$
 \dist (q,x)<\e, \quad
 \dist (f^{M_\e+1+i}q, x_0 ) \le \e, \;\; i=0, \dots, k-1,
 \quad \text{and }\; f^{2M_\e+k+1}q=q.
$$
Clearly, $q\ne x_0$.
Let $q'=f^{M_\e+1}q$. Since the function $\a$ is Lipschitz
and $\a(x_0)=1$, it follows from Corollary \ref{close orbits prod}
that there exists a constant $\gamma$ independent of $k$ 
such that
$$
\a(q')\a(fq')\dots \a(f^jq')
\ge e^{-\gamma\e} \quad\text{for }\; j=0, \dots, k-1.
$$
It follows  that
$$
\begin{aligned}
 &c(q)/\b \; \ge \;\a^*(q,2M_\e+k+1) \,\ge  \\
  & \a(f^{2M_\e+k}q) \dots \a(f^{M_\e+k}q)  +\dots  + 
   \a(f^{2M_\e+k}q)\dots \a(f^{M_\e+1}q) = \\ 
  &  \a(f^{2M_\e+k}q) \dots \a(f^{M_\e+k+1}q) \cdot
  \left( \a(f^{k-1}q') + \dots +\a(f^{k-1}q')\cdots \a(q')\right) \ge 
m^{M_\e} k \, e^{-\gamma\e}, 
 \end{aligned} 
$$
where $m=\min_\M \a(x)$. 
Taking a sufficiently large $k$ ensures that $c(q)>N$.
\end{proof}

Let $\V\ne \E_1$ be a continuous $A$-invariant sub-bundle
and let $x\ne x_0$ be a point such that $\V(x)\ne \E_1(x)$.
Then for every periodic point $p$ in a small neighborhood of $x$,
$\V(p)\ne \E_1(p)$ and hence $\V(p)$ is spanned 
by ${\mathbf v}(p)$ as in \eqref{v(p)}. It follows from Lemma~\ref{c(p)}
and continuity of $\V$ that $\V(x)=\E_1(x)$, a contradiction.
\vskip.3cm

{\bf (iii)} 
We describe an example similar to one in  \cite{Guy}.
Let $f:\M\to\M$ be an Anosov diffeomorphism and let $S$ 
be a closed $f$-invariant set in $\M$ that does not contain 
periodic points.  Let $\a$ be a smooth function such that
$$ 
   \a(x)=1\;\text{ for }\;x\in S \quad \text{and}\quad
   0< \a(x)<1 \;\text{ for }\;x\notin S.
$$   
At every periodic point $p=f^np$ the matrices 
$\A(p,n)$ and $\B(p,n)$ have the same eigenvalues,
1 and $\a^\times(p,n)<1$, are hence are conjugate.

However, there is no continuous function $C$
such that $A(x)=C(fx)B(x)C^{-1}(x)$.
Otherwise, for $x\in S$
$$
  \A(x,n)=
  \left[ \begin{array}{cc} 1 & n\b \\ 0 & 1 \end{array} \right] =
  C(f^n x)\B(x,n)C(x)^{-1}=C(f^n x)C(x)^{-1},
 $$
 which implies that $C$ is unbounded. 
 
 It can be shown as in (i) that the cocycles $A$ and $B$ 
 are measurably comologous. It can also be seen as in (ii)
 that the set of conjugacies $C(p)$ at the periodic points
 is unbounded, unlike in our Example  \ref{periodic conj}.
 $\QED$


\end{document}